\documentclass[a4paper,12pt]{amsart}
\usepackage{amsmath,amsthm,amssymb,amscd,amsfonts}
\usepackage[all]{xy}
\usepackage{enumerate}
\usepackage{bm}
\usepackage{ulem}
\usepackage{graphicx}
\usepackage[dvipdfmx]{xcolor}
\usepackage{cleveref}
\usepackage{tikz}
\usepackage{tikz-cd}
\theoremstyle{plain}
\newtheorem{thrm}{Theorem}
\newtheorem{thm}{Theorem}[section]
\newtheorem{prop}[thm]{Proposition}
\newtheorem{cor}[thm]{Corollary}
\newtheorem{lem}[thm]{Lemma}

\theoremstyle{definition}

\newtheorem{rem}[thm]{Remark}
\numberwithin{equation}{section}
\newcommand{\Z}{\mathbb{Z}}
\newcommand{\Q}{\mathbb{Q}}
\newcommand{\Sym}{\mathfrak S}
\newcommand{\hooklongrightarrow}{\lhook\joinrel\longrightarrow}
\DeclareMathOperator{\rk}{rk}
\DeclareMathOperator{\gr}{gr}
\DeclareMathOperator{\Lie}{\mathcal L}
\DeclareMathOperator{\Der}{Der}
\DeclareMathOperator{\End}{End}
\DeclareMathOperator{\Hom}{Hom}
\DeclareMathOperator{\Aut}{Aut}
\begin{document}
\title{The Andreadakis Problem for the McCool groups}
\author{Jaques Darn\'{e}}
\address{Jaques Darn\'{e}; LAMFA, Universit\'{e} de Picardie Jules Verne}
\email{jacques.darne@u-picardie.fr} 
\author{Naoya Enomoto}
\address{Naoya Enomoto; The University of Electro-Communications}
\email{enomoto-naoya@uec.ac.jp}
\author{Takao Satoh}
\address{Takao Satoh; Department of Mathematics, Faculty of Science Division II, Tokyo University of Science}
\email{takao@rs.tus.ac.jp}
\begin{abstract}
In this short paper, we show that the McCool group does not satisfy the Andreadakis equality from degree $7$, and we give a lower bound for the size of the difference between the two relevant filtrations.
As a consequence, we see that the Andreadakis problem for the McCool group does not stabilize.
\end{abstract}
\maketitle


\section{Introduction}

Let $G$ be a group, and $\mathrm{Aut}(G)$ its automorphism group.
Let $G=\Gamma_1(G) \supset \Gamma_2(G) \supset \cdots$ be the lower central series of $G$.
The action of $\mathrm{Aut}(G)$ on the $k$-th nilpotent quotient group $G/\Gamma_{k+1}(G)$ corresponds to a homomorphism
$\mathrm{Aut}(G) \rightarrow \mathrm{Aut}\,(G/\Gamma_{k+1}(G))$, whose kernel is denoted by $\mathcal{A}_k(G)$,
or simply $\mathcal{A}_k$.
This gives a descending filtration
\[ \mathrm{Aut}(G) \supset \mathcal{A}_1(G) \supset \mathcal{A}_2(G) \supset \cdots \]
of $\mathrm{Aut}\,G$, called the \textit{Andreadakis-Johnson filtration} of $\mathrm{Aut}(G)$.
The first subgroup $\mathcal{A}_1(G)$ is called the
IA-automorphism group of $G$, and is denoted by $\mathrm{IA}(G)$.
Andreadakis \cite{And} showed that this filtration is an N-series. 
Hence, each of the graded quotient $\gr_k(\mathcal{A}_*(G)):=\mathcal{A}_k(G)/\mathcal{A}_{k+1}(G)$
is an abelian group, and their direct sum $\gr(\mathcal{A}_*(G))$ is a graded Lie algebra.
Since the Andreadakis-Johnson filtration is an N-series we have $\Gamma_k(\mathrm{IA}(G)) \subset \mathcal{A}_k(G)$
for any $k \geq 1$.
It is a natural question to ask: What is the difference between $\Gamma_k(\mathrm{IA}(G))$ and $\mathcal{A}_k(G)$?
The question of whether or not $\Gamma_k(\mathrm{IA}(G))=\mathcal{A}_k(G)$ for any $k \geq 1$ is called the \textit{Andreadakis problem} for
$\mathrm{Aut}(G)$.

\medskip

Andreadakis was mostly interested in the free group case.
Let $F_n$ be the free group of rank $n$ with basis $x_1, \ldots, x_n$.
Andreadakis \cite{And} showed that each of $\mathcal{A}_k(F_n)/\mathcal{A}_{k+1}(F_n)$ is a free abelian group of finite rank,
and gave its rank for $k=1$ by using Magnus's generators of $\mathrm{IA}(F_n)$.
He also showed that $\Gamma_k(\mathrm{IA}(F_2))=\mathcal{A}_k(F_2)$ for every $k \geq 1$, and
$\Gamma_k(\mathrm{IA}(F_3)) = \mathcal{A}_k(F_3)$ for $1 \leq k \leq 3$.
Then he conjectured that $\Gamma_k(\mathrm{IA}(F_n)) =\mathcal{A}_k(F_n)$ for any $n \geq 3$ and $k \geq 1$.
Today, this is known as \textit{the Andreadakis conjecture}.
Bachmuth \cite{Ba} showed that $\Gamma_2(\mathrm{IA}(F_n)) =\mathcal{A}_2(F_n)$ for all $n \geq 2$.
This fact is also obtained from the independent works by 
Cohen-Pakianathan \cite{CP1, CP2}, Farb \cite{Fa} and Kawazumi \cite{Ka} who determined the abelianization of $\mathrm{IA}(F_n)$
by using the first Johnson homomorphism.
Satoh \cite{S2} showed that $\Gamma_3(\mathrm{IA}(F_n)) =\mathcal{A}_3(F_n)$ for every $n \geq 3$ by using the second Johnson
homomorphism and straightforward calculations.
In general, however, the Andreadakis conjecture does not hold.
Bartholdi \cite{LB1, LB2} showed that
$\mathcal{A}_{4}(F_3)/\Gamma_4(\mathrm{IA}(F_3)) \cong (\Z/2\Z)^{\oplus 14} \oplus (\Z/3\Z)^{\oplus 9}$ and
$\mathcal{A}_{5}(F_3)/\Gamma_5(\mathrm{IA}(F_3)) \otimes_{\Z} \Q \cong \Q^{\oplus 3}$, using computer calculations.
For a general $n \geq 4$, the conjecture is still open. In particular, it is open in the stable range (for $n \gg k$). In fact, the natural map 
$\Gamma_k(\mathrm{IA}(F_n))/\Gamma_{k+1}(\mathrm{IA}(F_n)) \rightarrow \mathcal{A}_{k}(F_n)/\mathcal{A}_{k+1}(F_n)$ has been shown to be surjective for $n \geq k+2$, rationally by Massuyeau and Sakasai~\cite{MS}, and over the integers by Darn\'{e}~\cite{Da1}.

\medskip

We can consider the restriction of the Andreadakis problem to subgroups of $\mathrm{IA}(G)$, for a given group $G$.
Precisely, given a subgroup $N$ of $\mathrm{IA}(G)$, we can ask whether or not the inclusions $\Gamma_k(N) \subset \mathcal{A}_k(G) \cap N$ are equalities.
The Andreadakis problem has been studied for several subgroups of $\mathrm{Aut}(F_n)$.
One of the most important subgroups is the Torelli subgroup $\mathcal{I}_{g,1}$ of the mapping class group
of a compact oriented surface of genus $g$ with one boundary component. It is considered as a subgroup of $\mathrm{Aut}(F_{2g})$
thanks to Dehn and Nielsen's classical work. Johnson \cite{Jo} determined the
abelianization of the Torelli group, and showed that it has many direct summands isomorphic to
$\Z/2\Z$ by using the Birman-Craggs homomorphism. From this, it immediately follows that
$\Gamma_2(\mathcal{I}_{g,1}) \subsetneqq \mathcal{A}_2(F_{2g}) \cap \mathcal{I}_{g,1}$.
Morita \cite{Mo} showed that
$(\mathcal{A}_3(F_{2g}) \cap \mathcal{I}_{g,1})/\Gamma_3(\mathcal{I}_{g,1})) \otimes_{\Z} \Q$ is not trivial for $g \geq 3$
by using the Casson invariants. Thus, the Andreadakis problem for the Torelli group never holds.
However, $\mathrm{Aut}(F_n)$ contains other subgroups of topological significance. First among these is the braid group $B_n$ on $n$-strands,
which embeds into $\mathrm{Aut}(F_n)$ via the \textit{Artin action} of $B_n$ on $F_n$, a fact that has been known since the introduction of the braid group by Artin~\cite{Art}.
The subgroup $\mathcal{A}_1(F_n) \cap B_n$ is the pure braid group, denoted by $P_n$.
Darn\'{e} \cite{Da2} proved that $P_n$ satisfies the Andreadakis equality, that is: 
$\Gamma_k(P_n) = \mathcal{A}_k(F_n) \cap P_n$ for any $n \geq 2$ and $k \geq 1$.

\medskip

In this paper, we study another subgroup of topological significance, which is the \textit{basis-conjugating automorphism group} of $F_n$.
An automorphism of $F_n$ is called \textit{basis-conjugating} (with respect to our fixed basis $x_1, ..., x_n$) if it sends each element $x_i$ to one of its conjugates. Such automorphisms form a subgroup $\mathrm{P}\Sigma_n$ of $\mathrm{IA}(F_n)$. It identifies with the~\textit{pure loop braid group} (also called the~\textit{pure welded braid group}), acting faithfully on $F_n$ via the Artin action~\cite{Da}.
Since McCool \cite{McC} gave the first finite presentation for $\mathrm{P}\Sigma_n$, it is also called the McCool group. 
In \cite{ES}, the second and third authors proved that $\Gamma_4(\mathrm{P}\Sigma_n)=\mathcal{A}_4(F_n) \cap \mathrm{P}\Sigma_n$ 
for any $n \geqq 3$.
Our main theorem is the following.
\begin{thrm}\label{main_thm}
For any $n \geq 3$, $\mathrm{P}\Sigma_n$ does not satisfy the Andreadakis equality. In fact:
\[\text{The quotient } \frac{\mathcal{A}_7(F_n) \cap \mathrm{P}\Sigma_n}{\Gamma_7(\mathrm{P}\Sigma_n)} \text{ has a central subgroup which is free abelian of rank } \binom{n}{3}.\]
\end{thrm}
This relies on the determination of a counter-example for $n = 3$, which (contrary to what happens in the case of $\mathrm{IA}(F_n)$ and the stable problem mentioned above) propagates to higher values of $n$.

\section{Johnson homomorphisms}
\label{sec_Johnson}

In this section, we recall the definition of the Johnson homomorphism for subgroups of $\mathrm{Aut}(F_n)$, and we recall how to reformulate the Andreadakis problem in terms of this morphism. We also introduce tangential derivations.

\vspace{0.5em}

For a group $G$, denote by $\Lie_k(G) := \Gamma_k(G)/\Gamma_{k+1}(G)$ the $k$-th graded quotient of the lower central series of $G$,
and by $\Lie_*(G) := {\bigoplus}_{k \geq 1} \Lie_k(G)$ the associated graded sum.
The graded sum $\Lie_*(G)$ is endowed with a natural graded Lie algebra structure induced by the commutator bracket in $G$.
It is always generated by its degree one part $\Lie_1(G) = G^{ab}$. If $G$ is the free group $F_n$, then $\Lie_*(F_n)$ is the free Lie algebra on $H = F_n^{ab} \cong \Z^n$, denoted by $\mathbb L[n]$. If $k \geq 1$, its degree $k$ part is denoted by $\mathbb L_k[n]$.




Let us recall some facts about the derivation algebra of the free Lie algebra $\mathbb L[n]$.
If $\mathfrak g$ is a Lie algebra, then its Lie algebra of derivations is: 
\[ \Der(\mathfrak g) := \{ f \in \End_\Z(\mathfrak g)
     \ |\ f([X,Y]) = [f(X),Y]+ [X,f(Y)] \,\, \text{for} \,\, X,Y \in \mathfrak g  \}. \]
When $\mathfrak g = \mathbb L[n]$ (which is free on $H = \Z^n$), then derivations (which correspond to sections of the canonical projection $\mathfrak g \rtimes \mathfrak g \twoheadrightarrow \mathfrak g$) are determined by their restriction to $H$. 
In particular, if $\Der_k(\mathbb L[n])$ denotes the subset of derivations raising the degree by $k \geq 0$, we have
\begin{align*} 
\Der_k(\mathbb L[n]) &= \{ f \in \Der(\mathbb L[n]) \ |\ f(x) \in \mathbb L_{k+1}[n] \,\,\text{for} \,\, x \in H \} \\
&\cong \Hom_\Z(H,\mathbb L_{k+1}[n]) \cong  H^* \otimes_\Z \mathbb L_{k+1}[n].
\end{align*}
and $\Der(\mathbb L[n])$ is a graded Lie algebra, for the decomposition:
\[ \Der(\mathbb L[n]) = \bigoplus_{k \geq 0} \Der_k(\mathbb L[n]) \cong \bigoplus_{k \geq 0} H^* \otimes_\Z \mathbb L_{k+1}[n].\]

The free Lie algebra $\mathbb L[n]$ injects into its universal enveloping algebra, which is the \textit{tensor algebra} $\mathbb T[n]$ (over~$\Z$), also known as the free associative ring on $n$ generators (that is, on $H = \Z^n$). Derivations of $\mathbb L[n]$ extend uniquely to \textit{derivations of the associative algebra} $\mathbb T[n]$, which are $\Z$-linear endomorphisms $f$ of $\mathbb T[n]$ satisfying $f(uv) = f(u)v + uf(v)$ for every $u,v \in \mathbb T[n]$. Indeed, derivations of $\mathbb T[n]$ are also uniquely determined by their restriction to $H$, so the inclusion of $\Der(\mathbb L[n])$ into $\Der(\mathbb T[n])$ can be seen as the direct sum of the inclusions of $H^* \otimes_\Z \mathbb L_{k+1}[n]$ into $H^* \otimes_\Z H^{\otimes k+1}$, over $k \geq 0$.



\medskip

Recall that we write simply $\mathcal A_*$ for the usual Andreadakis filtration $\mathcal A_*(F_n)$, and $\gr_k(\mathcal A_*)$ for $\mathcal A_k/\mathcal A_{k+1}$. The graded sum $\gr(\mathcal A_*) := \bigoplus_{k \geq 1} \gr_k (\mathcal A_*)$ is endowed with the Lie algebra structure whose Lie bracket is induced by commutators in $\Aut(F_n)$.
It is well-known that there is a well-defined injective morphism of graded Lie algebra
\[\tilde\tau : \gr(\mathcal A_*) \hooklongrightarrow \Der(\mathbb L[n]),\] described as follows: if $\overline \sigma$ denotes the class in $\gr_k(\mathcal A_*)$ of some $\sigma \in \mathcal A_k(F_n)$, and $\overline x$ denotes the class in $\Lie_j(F_n) \cong \mathbb L_j[n]$ of an element $x \in \Gamma_j(F_n)$,
\[\tilde\tau_k : \overline \sigma\ \longmapsto\ \left(\overline x \mapsto \overline{x^{-1} x^{\sigma}}\right) \in \Der_k(\mathbb L[n]) \cong H^* \otimes_\Z \mathbb L_{k+1}[n].\]
Notice that although the formula works for every $j \geq 1$, the latter isomorphism remembers only the case $j = 1$, which is the effect of the derivation on $H$.

\medskip

If $N$ is a subgroup of $\mathrm{IA}(F_n)$, then the canonical morphism from $\mathcal A_k \cap N$ to $\gr_k(\mathcal A_*)$ has $\mathcal A_{k+1}\cap N$ as its kernel, so $\gr(\mathcal A_* \cap N)$ injects into $\gr(\mathcal A_*)$. The restriction of $\tilde\tau$ to this Lie subalgebra $\gr(\mathcal A_* \cap N)$ is denoted by $\tilde\tau^N$. The inclusions $\Gamma_k(N) \subseteq \mathcal A_k \cap N$ induce a morphism of graded Lie algebras:
\[\iota_*^N: \Lie_*(N) \rightarrow \gr(\mathcal A_* \cap N).\]
We denote by $\tau^N$ the composition $\tilde\tau^N \circ \iota_*^N$. We sum this up in the commutative diagram:
\[
\xymatrix{
\Lie_k(N) \ar[r]^{\iota_*^N} \ar@/_1.3pc/[rr]_{\tau_k^N}
&\gr(\mathcal A_* \cap N)  \ar@{^{(}->}[r]^{\tilde\tau_k^N}
&\Der_k(\mathbb L[n])
}
\]
When $N = \mathrm{IA}(F_n)$, the homomorphisms $\tilde\tau_k$ and $\tau_k$ are the usual \textit{Johnson homomorphisms} of $\mathrm{Aut}(F_n)$. We will still use this name for $\tilde\tau_k^N$ and $\tau_k^N$, and we reserve the right to loose the superscript and the indices whenever convenient.

\medskip

Since $\tilde\tau_k^N$ is injective, we have $\ker(\iota_k^N) = \ker(\tau_k^N)$. We denote this kernel by $\kappa_k(N)$. We easily see that:
\begin{equation}\label{eq_JohnsonKernel}
\kappa_k(N) = \ker\left(\iota_k^N: \Gamma_k(N)/\Gamma_{k+1}(N) \rightarrow \mathcal A_k/\mathcal A_{k+1}\right) = (\Gamma_k(N) \cap \mathcal A_{k+1})/\Gamma_{k+1}(N).
\end{equation}
This is why $\kappa_k(N) = \ker(\tau_k^N)$ is relevant to the study of the Andreadakis equality for $N$
(whence our choice of the simpler notation $\tau_k$, and not $\tilde\tau_k$, for this Johnson homomorphism). 
In fact, we can reformulate the Andreadakis problem for $N$ using $\tau_k^N$ as follows: 
\begin{prop}\label{P-And}
Fix $k_0 \geq 1$. Then $\Gamma_k(N) = \mathcal A_k \cap N$ for all $k \leq k_0$ if and only if $\tau_k^N$ is injective for all $k < k_0$.
\end{prop}

\begin{proof}
Since $\ker(\tau_k^N) = (\Gamma_k(N) \cap \mathcal A_{k+1})/\Gamma_{k+1}(N)$, $\tau_k^N$ is certainly injective when $\Gamma_{k+1}(N) = \mathcal A_{k+1} \cap N$. Conversely, if the inclusion $\Gamma_j(N) \subseteq \mathcal A_j \cap N$ is not an equality (for some $j \geq 1$, we get an element $x \in (\mathcal A_j \cap N) \setminus \Gamma_j(N)$. Let $1 \leq k < j$ be maximal such that $x \in \Gamma_k(N)$. Then $x \notin \Gamma_{k+1}(N)$, but $x \in \mathcal A_j \subseteq \mathcal A_{k+1}$. This gives a non-trivial element $\overline x$ of $\ker(\tau_k^N) = (\Gamma_k(N) \cap \mathcal A_{k+1})/\Gamma_{k+1}(N)$, for \textit{some} $k < j$. The conclusion follows.
\end{proof}

Here we focus on the case where $N$ in the subgroup $\mathrm P\Sigma_n$ of (pure) basis-conjugating automorphisms of $\mathrm{IA}(F_n)$. An automorphism $\sigma \in \mathrm P\Sigma_n$ sends each $x_i$ to some conjugate $w_i^{-1}x_i w_i$. So if $\sigma \in \Gamma_k(\mathrm P\Sigma_n)$, then the formula defining the Johnson homomorphism gives:
\[\tau_k(\overline \sigma)(\overline{x_i}) = \overline{x_i^{-1}x^{\sigma}} 
= \overline{[x_i, w_i]} = [\overline{x_i}, \overline{w_i}] =[X_i, W_i].\]
In other words, the derivation $\tau_k(\overline \sigma)$ sends each element $X_i := \overline{x_i} \in H$ of the basis of $\mathbb L[n]$ to a bracket of the form $[X_i, W_i]$. Such derivations are called \textit{tangential}, and they are easily seen to form a Lie subalgebra of $\Der(\mathbb L[n])$, which we denote by $\Der^\tau(\mathbb L[n])$. It is also the one denoted by $\mathfrak{p}_n$ in \cite{S1} and \cite{ES}.

\section{A counter-example for $n = 3$}

By Proposition {\rmfamily \ref{P-And}}, in order to disprove the Andreadakis equality for $\mathrm{P}\Sigma_3$ for $k = 7$, it suffices to show that $\tau_6^{\mathrm{P}\Sigma_3}$ is not injective.

\subsection{The structure of $\mathrm P\Sigma_3$}

It turns out that the structure of $\mathrm P\Sigma_3$ and of its lower central series can be completely described, thanks to \cite[\S A.1]{Ib}. We do it in some detail, for the convenience of the reader. The main point is that the quotient of $\mathrm P\Sigma_3$ by inner automorphisms is free, a fact that fails for $\mathrm P\Sigma_n$ as soon as $n > 3$. 

\medskip

A classical result of McCool \cite{McC} asserts that the group $\mathrm P\Sigma_n$ of (pure) basis-conjugating automorphisms of the free group $F_n$ is the group generated by generators  $K_{ij}\ (i \neq j)$ (the automorphism of $F_n$ sending $x_i$ to ${x_j}^{-1} x_i x_j$ and fixing the $x_t$ for $t \neq i$),  submitted to the \textit{McCool relations}:
\[[K_{ik}, K_{jk}] = [K_{ik}K_{jk}, K_{ij}] =
[K_{ij},K_{kl}] = 1 \ \ \text{ for $i,j,k,l$ pairwise distinct.}\]
For $n = 3$, we get six generators $K_{21}$, $K_{12}$, $K_{13}$, $K_{31}$, $K_{23}$ and $K_{32}$, submitted to nine relations. Then, let us remark that in this special case, the element $K_{ik}K_{jk}$ (which is equal to $K_{jk}K_{ik}$) is the inner automorphism $c_k : w \mapsto x_k w x_k^{-1}$. We can re-write the presentation in terms of the generators $c_1$, $c_2$, $c_3$, $K_{21}$, $K_{21}$ and $K_{13}$ (using $K_{31} = c_1 K_{21}^{-1}$, etc.) as:
\[[c_j, K_{ij}] = [c_k, K_{ij}] =  1 \ \ \text{ for $\{i,j,k\} = \{1,2,3\}$.}\]
Notice that all these twelve relations hold, but we can deduce them from the nine involving only the six generators $c_1$, $c_2$, $c_3$, $K_{21}$, $K_{12}$ and $K_{13}$. 

\medskip

Now, let us consider the subgroup $F = \langle c_1, c_2, c_3\rangle$. It is the subgroup of inner automorphisms of $F_3$, so it is free on its generators $c_k$, and it is normal in $\Aut(F_3)$, hence in $\mathrm P\Sigma_3$ (this last statement can also easily be deduced from the above presentation). The quotient $\mathrm P\Sigma_3/F$ has the presentation given by adding the relations $c_k = 1$ to the above presentation. Using this relation to eliminate the $c_k$, we see that all the relations become trivial, so $\mathrm P\Sigma_3/F$ is free on the classes of $K_{21}$, $K_{12}$ and $K_{13}$ (denoted respectively by $\alpha$, $\beta$ and $\gamma$). We thus get a short exact sequence of groups:
\[F_3 \hookrightarrow \mathrm P\Sigma_3 \twoheadrightarrow F_3.\]
Since the quotient is free, this sequence automatically splits : consider, for instance, the section sending $\alpha$ to $K_{21}$, $\beta$ to $K_{12}$ and $\gamma$ to $K_{13}$. This gives a semi-direct product decomposition:
\[\mathrm P\Sigma_3 \cong F_3 \rtimes F_3.\]
Set $\mathfrak{h} := \mathcal{L}_*(\langle c_1, c_2, c_3 \rangle) \cong \mathbb L[3]$ and
$\mathfrak{g} := \mathcal{L}_*(\langle \alpha, \beta, \gamma \rangle) \cong \mathbb L[3]$.

\subsection{The lower central series}

Recall that the conjugation action of $\Aut(F_3)$ on its subgroup of inner automorphisms coincides with its canonical action on $F_3$. Since the above action is via a subgroup of
$\mathrm{IA}(F_3)$,
this is an almost-direct product, so by a theorem of Falk and Randell~\cite[Th.~3.1]{FR} (see also~\cite[Def-prop.~3.5]{Da2})
it induces a decomposition of the Lie algebra $\Lie_*(\mathrm P\Sigma_3)$ associated to the lower central series of $\mathrm P\Sigma_3$
as a semi-direct product of Lie algebras:
\begin{equation}\label{sdprod_dec}
\Lie_*(\mathrm P\Sigma_3) \cong \mathfrak h \rtimes \mathfrak g.
\end{equation}

This implies that $\Lie_*(\mathrm P\Sigma_3)$ has no torsion, and that for any $k \geq 1$
\[ \rk(\Lie_k(\mathrm P\Sigma_3)) = 2 \rk(\mathbb L_k[3]).\]
\begin{rem}
All the results of~\cite{MP} can be easily recovered from this decomposition.
\end{rem}

\subsection{The Andreadakis problem}\label{sec_counter-ex}

The Andreadakis problem for $\mathrm P\Sigma_3$ is equivalent to the injectivity of the Johnson morphism $\tau : \Lie_*(\mathrm P\Sigma_3) \rightarrow \Der(\mathbb L[3])$, characterized by:
\[d_{ij} := \tau(\overline K_{ij}) :
\begin{cases}
    X_i \mapsto [X_j, X_i]\\
    X_k \mapsto X_k \text{ if } k \neq i.
\end{cases}\]
Let us first study the behaviour of its kernel $\kappa(\mathrm P\Sigma_3)$ with respect to the semi-direct product decomposition~\eqref{sdprod_dec}. It is easy to show that $\tau|_{\mathfrak h}$ is injective: $\mathfrak h$ is sent to the Lie algebra of inner derivations of $\mathbb L[3]$,
and this must be an isomorphism (this is the Andreadakis equality for inner automorphisms of $\mathbb L[3]$, see~\cite[\S~2.2.2]{Da3}).
Thus, $\kappa(\mathrm P\Sigma_3) \cap \mathfrak h = 0$. We can in fact be more precise:

\begin{prop}\label{kappa_in_g}
We have $\kappa(\mathrm P\Sigma_3) \subseteq \mathfrak g$.
\end{prop}

\begin{proof}
We first show that $\tau(\mathfrak h) \cap \tau(\mathfrak g) = 0$. On the one hand, recall \cite[\S~2.2.2]{Da3} that $\tau(\mathfrak h)$ is the set $\mathfrak I$ of inner derivations of $\mathbb L[3]$. On the other hand, $\tau(\mathfrak g)$ is generated by $d_{12}$, $d_{21}$ and $d_{13}$, all of which send $X_3$ to $0$, so it is included in the Lie subalgebra $\mathfrak d_3$ of $\Der(\mathbb L[3])$ consisting of derivations sending $X_3$ to $0$. Then $\mathfrak I \cap \mathfrak d_3$ is the set of inner derivations $X \mapsto [W, X]$ such that $[W, X_3] = 0$, which is equivalent to $W \in \Z X_3$ (see for instance~\cite[Lem.~3.6]{Da3}). Hence $\mathfrak I \cap \mathfrak d_3$ is linearly generated by $[X_3,-] = d_{13} + d_{23}$. But no multiple of $[X_3,-]$ belongs to $\tau(\mathfrak g)$, since its degree one is the linear span of $d_{12}$, $d_{21}$ and $d_{13}$, and the $d_{ij}$ are linearly independent. This shows that $\tau(\mathfrak h) \cap \tau(\mathfrak g)$ is trivial.

Linearly, $\mathfrak h \rtimes \mathfrak g$ is just the direct sum $\mathfrak h \oplus \mathfrak g$. Let $u+v \in \kappa(\mathrm P\Sigma_n)$ with  $u \in \mathfrak h$ and $v \in \mathfrak g$. Then $\tau(u+v) = 0$ is equivalent to $\tau(u) = \tau(-v) \in \tau(\mathfrak h) \cap \tau(\mathfrak g) = 0$. But $\tau|_{\mathfrak h}$ is injective, so $\tau(u) = 0$ means that $u = 0$, hence $u+v = v \in \mathfrak g$.
\end{proof}

\begin{rem}
The above statement can seem surprising, since $\mathfrak g$ depends on the choice of section of $\mathrm P\Sigma_3 \twoheadrightarrow \mathrm P\Sigma_3/F$ that we made, whereas $\kappa(\mathrm P\Sigma_3)$ does not. However, we will see later (Proposition~\ref{kappa_as_intersection}), that this inclusion stays true for all the obvious choices of sections.
\end{rem}

The Lie algebra $\mathfrak g$ is free on the three elements $a:= \overline{K_{12}}$, $b:= \overline{K_{21}}$ and $c:= \overline{K_{13}}$ of $\Lie_1( \mathrm P\Sigma_3) = \mathrm P\Sigma_3^{ab}$. An explicit computation leads us to the following left-normed element:
\begin{align*}
\varpi:= &\hspace{1.2em}  [c, a, b, b, a, c] + [c, a, c, b, a, b] + [c, b, a, c, a, b] + [c, b, a, c, b, a] \\ 
&+ [c, b, b, a, a, c] + [c, b, c, b, a, a] - [c, a, c, a, b, b] - [c, a, c, b, b, a] \\
&- [c, b, a, a, c, b] - [c, b, a, b, a, c] - [c, b, b, a, c, a] - [c, b, c, a, b, a].
\end{align*}

Our main result is:
\begin{prop}
The element $\varpi$ is a non-trivial element of the kernel of $\tau$.
\end{prop}

\begin{proof}
The non-triviality of $\varpi$ in the free Lie algebra can be seen directly by noticing that it has the non-zero term $c^2b^2a^2$
as an element of the tensor algebra on $a, b, c$, which is the universal enveloping algebra of the free Lie algebra generated by $a,b,c$. This term only appears from the bracket $[c, a, c, a, b, b]$ is our expression (the only bracket where the $a's$ appear before the $b's$). Proving that $\tau(\varpi) = 0$ can be done by a direct computation, although it is easier to do it by computer calculation than by hand. Precisely, since $\tau$ is a Lie morphism, one needs to replace $a$ by $\tau(a) = d_{12}$, $b$ by $\tau(b) = d_{21}$ and $c$ by $\tau(c) = d_{13}$ in the previous expression, and to check that the derivation $\tau(\varpi)$ so obtained is trivial on the generators $X_1$, $X_2$, $X_3$ of $\mathbb L[3]$. Notice that this can be done in the Lie algebra of derivations of the tensor algebra $\mathbb T[3]$, if one prefers using polynomials instead of Lie brackets (see~\cref{sec_Johnson}).
\end{proof}



\section{Functoriality and propagation of the counter-example}

Contrary to what could happen with $\mathrm{IA}_n$, the Andreadakis problem for $\mathrm P\Sigma_n$ does not stabilize. The main reason for this is the existence of a splitting of the canonical injection $\mathrm P\Sigma_n \hookrightarrow \mathrm P\Sigma_{n+1}$. Indeed, this injection corresponds to adding a strand to (pure) loop braids, and a splitting is given by forgetting a strand.

Precisely, let $n \geq 3$.
For a three points subset
$I = \{i_1 < i_2 < i_3\} \subset \{1, ..., n\}$, we define the injection $\eta=\eta_I : F_3 \rightarrow F_n$
by $x_s$ to $x_{i_s}$ for $1 \leq s \leq 3$.
Then we get an injection of $\Aut(F_3)$ into $\Aut(F_n)$ by sending $f$ to the automorphism sending $x_{i_s}$ to $\eta(f(x_s))$
for $1 \leq s \leq 3$ and fixing the $x_j$ for $j \notin I$. This construction clearly sends basis-conjugating automorphisms to basis-conjugating automorphisms, so it induces an injection:
\[\iota_I : \mathrm P\Sigma_3 \hookrightarrow \mathrm P\Sigma_n.\]
This injection \textit{splits}: there is a projection 
\[\pi_I : \mathrm P\Sigma_n \twoheadrightarrow \mathrm P\Sigma_3\]
such that $\pi_I \circ \iota_I = id$. Namely, notice that an element $f \in \mathrm P\Sigma_n$ sends each $x_j$ to a conjugate of itself, so it preserves the normal closure $\mathcal N_I := \mathcal N(\{x_j\}_{j \notin I})$ (where $\mathcal N(A)$ denotes the normal closure in $F_n$ of a subset $A$ of $F_n$). As a consequence, it induces an automorphism $\pi_I(f)$ of $F_3 \cong F_n/\mathcal N_I$ (the fact that $\pi_I(f)$ is invertible when $f$ comes from the same construction for $f^{-1}$). The remaining verifications are left to the reader.

\begin{rem}
With the point of view of pure loop braids, $\pi_I$ corresponds to forgetting the strands not indexed by an element of $I$.
\end{rem}

Since the construction $\Lie_*(-)$ is functorial (it takes a group morphism to a Lie algebra morphism), this induces a split injection of Lie rings (still denoted by $\iota_I$ and $\pi_I$):
\[\xymatrix{
 \Lie_*(\mathrm P\Sigma_3) \ar@{^{(}->}@<.5ex>[r]^{\iota_I}
 &\Lie_*(\mathrm P\Sigma_{n}) \ar@{->>}@<.5ex>[l]^{\pi_I}.
 }
\]

It was recalled at the end of \cref{sec_Johnson} that the Johnson morphism for $\mathrm P\Sigma_n$ takes values in the Lie algebra $\Der^\tau(\mathbb L[n])$ of tangential derivations.
We show the story that we have just told for McCool groups also holds for these. Precisely, the inclusion of $I$ into $\{1, ..., n\}$ induces an injection of $\mathbb L[3]$ into $\mathbb L[n]$ (sending $X_s$ to $X_{i_s}$ for $1 \leq s \leq 3$), also denoted by $\eta$. It then induces an injection
\[\iota_I : \Der^\tau(\mathbb L[3]) \hookrightarrow \Der^\tau(\mathbb L[n])\]
defined, for $d \in \Der^\tau(\mathbb L[3])$, by $\iota_I(d) : X_{i_s} \mapsto \eta(d(X_s))$ for $1 \leq s \leq 3$
and $\iota_I(d)(X_j) = 0$ for $j \notin I$. Conversely, similarly to what happened for basis-conjugating automorphisms, if we take $d \in \Der^\tau(\mathbb L[n])$, then it preserves the Lie ideal $\mathcal J$ of $\mathbb L[n]$ generated by the $X_j$ for $j \notin I$, so it induces a derivation $\pi_I(d)$ of $\mathbb L[3] = \mathbb L[n]/\mathcal J$. The resulting map $\pi_I$ is then a splitting of $\iota_I$. 

\medskip

All these maps fit in the following commutative diagram:
\[\xymatrix{
 \kappa_*(\mathrm P\Sigma_3) \ar@{^{(}->}[r] \ar@{^{(}-->}[d]^{\iota_I}
 &\Lie_*(\mathrm P\Sigma_3)
 \ar@{^{(}->}[d]^{\iota_I} \ar[r]^(.45){\tau}
 &\Der^{\tau}(\mathbb L[3]) \ar@{^{(}->}[d]^{\iota_I}
 \cr
 \kappa_*(\mathrm P\Sigma_n)
 \ar@{^{(}->}[r]
 \ar@{-->>}@<1ex>[u]^{\pi_I}
 &\Lie_*(\mathrm P\Sigma_n)
 \ar@{->>}@<1ex>[u]^{\pi_I}
 \ar[r]^(.45){\tau}
 &\Der^{\tau}(\mathbb L[n]).
 \ar@{->>}@<1ex>[u]^{\pi_I}
}
\]
where $\tau$ is the Johnson morphism, $\kappa$ is its kernel (as recalled in~\cref{sec_Johnson}), and ``commutative" means that every square commutes. Thus we get induced maps on the kernels, still denoted by $\iota_I$ and $\pi_I$, still satisfying $\pi_I \circ \iota_I = id$, since they are the restrictions to the kernels of the previous $\iota_I$ and $\pi_I$.

\medskip

The element $\varpi$ from \cref{sec_counter-ex} is a non-trivial element of $\kappa_6(\mathrm P\Sigma_3)$, so for each $n \geq 3$ we get a collection of elements $\iota_I(\varpi)$ in $\kappa_6(\mathrm P\Sigma_n)$ indexed by subsets of cardinality $3$ of $\{1, ..., n\}$. All of these are non-trivial, since $\pi_I(\iota_I(\varpi)) = \varpi \neq 0$. These elements give counter-examples to the Andreadakis equality for $\mathrm P\Sigma_n$ for every $n \geq 3$. Furthermore, we are going to show that they are linearly independant, giving a lower bound on the size of $\kappa_6(\mathrm P\Sigma_n)$, hence of $(\mathrm P\Sigma_n \cap \mathcal{A}_7)/\Gamma_7(\mathrm{P}\Sigma_n)$, as announced. This relies on the following
\begin{lem}
Let $n \geq 3$ be an integer. If $I$ and $J$ are distinct subsets of $\{1, ..., n\}$ of cardinality $3$, then $\pi_J(\iota_I(\varpi)) = 0$.
\end{lem}

\begin{proof}
Recall that here $\pi_J \circ \iota_I$ denotes the restriction to $\kappa$ of the Lie algebra morphism $\pi_J \circ \iota_I$ from $\Lie_*(\mathrm P\Sigma_3)$ to itself. Given the definition of $\varpi$, we only need to show that $\pi_J \circ \iota_I$ sends either $a$, $b$ or $c$ to $0$. Recall further that this Lie morphism is induced by the group morphism $\pi_J \circ \iota_I$ from $\mathrm P\Sigma_3$ to itself described above. Given the definition of $a$, $b$ and $c$, it will be enough to show that this morphism sends $K_{12}$, $K_{21}$ or $K_{13}$ to the identity. The latter boils down to observing that one of them becomes the identity if either one of $x_1$, $x_2$ or $x_3$ is equaled to $1$. For instance, if $I = \{i_1 < i_2 < i_3\}$ and $J$ does no contain $i_3$, then $\iota_I(K_{13}) = K_{i_1i_3}$ (who sends $X_{i_1}$ to $X_{i_3}^{-1} X_{i_1} X_{i_3}$ and fixes the other generators) induces the identity on $F_n/\mathcal N_J$, since $X_{i_3} \in \mathcal N_J$. The other two cases are similar.
\end{proof}

\begin{cor}\label{lin_ind}
Let $n \geq 3$. The $\iota_I(\varpi)$, for $I \subseteq \{1, ..., n\}$ of cardinality $3$, are linearly independent elements of $\kappa_6(\mathrm P\Sigma_n)$
\end{cor}

\begin{proof}
If $\sum_I \lambda_I \iota_I(\varpi) = 0$, then for every $J$, apply $\pi_J$ to get $\lambda_J \varpi = 0$, hence $\lambda_J = 0$.
\end{proof}

\begin{proof}[Proof of Theorem~\ref{main_thm}]
Recall that $\kappa_6(\mathrm P\Sigma_n) \cong (\Gamma_6(\mathrm P\Sigma_n) \cap \mathcal A_7)/\Gamma_7(\mathrm P\Sigma_n)$ (this is formula \eqref{eq_JohnsonKernel} for $N = \mathrm P\Sigma_n$ and $k = 6$). This is a central subgroup of $(\mathrm P\Sigma_n \cap \mathcal A_7)/\Gamma_7(\mathrm P\Sigma_n)$, so the theorem follows from Lemma~\ref{lin_ind}.
\end{proof}

\section{Action of the symmetric group}

There is an action of the symmetric group $\Sym_3$ on $F_3$ permuting the variables. This induces an action on $\mathrm P\Sigma_3$ permuting the roles of the variables, which is the usual induced action on automorphisms, defined by $(\sigma \cdot f)(w) := \sigma \cdot f(\sigma^{-1} \cdot w))$, for $f \in \Aut(F_3)$, $w \in F_3$ and $\sigma \in \Sym_3$. These are actions by group automorphisms, so by functoriality of $\Lie_*(-)$, we get an induced action by graded Lie automorphisms on $\Lie_*(\mathrm P\Sigma_3)$. We also get an induced action on $\Lie(F_3) = \mathbb L[3]$ (permuting the $X_i$), and an induced action on $\Der(\mathbb L[3])$ which permutes the roles of the $X_i$, given by the same formula as for automorphisms. Then $\tau$ is easily checked to be equivariant, hence $\kappa(\mathrm P\Sigma_3) = \ker(\tau)$ is stable under this action. We first use this language to improve the result of Proposition~\ref{kappa_in_g}:
\begin{prop}\label{kappa_as_intersection}
Let $c = (123) \in \Sym_3$. We have $\kappa(\mathrm P\Sigma_3) = \mathfrak g \cap c \mathfrak g \cap c^2 \mathfrak g =  \bigcap\limits_{\sigma \in \Sym_3} \sigma \mathfrak g$.
\end{prop}

\begin{proof}
Since $\kappa(\mathrm P\Sigma_3)$ is stable under the action of $\Sym_3$ and included in $\mathfrak g$ (Proposition \ref{kappa_in_g}), it is included in every $\sigma \mathfrak g$ (for $\sigma \in \Sym_3$). Conversely, let us show that $\mathfrak g \cap c \mathfrak g \cap c^2 \mathfrak g \subseteq \kappa(\mathrm P\Sigma_3)$. Equivalently, we need to prove that $\tau(\mathfrak g \cap c \mathfrak g \cap c^2 \mathfrak g) = 0$. Let $\mathfrak d_i$ be the Lie subalgebra of $\Der(\mathbb L[3])$ consisting of derivations sending $X_i$ to $0$ (for $i = 1,2,3$).
We have remarked in the proof of Proposition~\ref{kappa_in_g} that $\tau(\mathfrak g) \subseteq \mathfrak d_3$. This implies that $\tau(c\mathfrak g) = c\tau(\mathfrak g) \subseteq c\mathfrak d_3 = \mathfrak d_1$. Similarly, $\tau(c^2\mathfrak g) \subseteq \mathfrak d_2$. As a consequence, we get $\tau(\mathfrak g \cap c \mathfrak g \cap c^2 \mathfrak g) \subseteq \mathfrak d_1 \cap \mathfrak d_2 \cap \mathfrak d_3 = 0$, since a derivation sending $X_1$, $X_2$ and $X_3$ to $0$ is trivial.
\end{proof}

We can describe explicitly the action of $\Sym_3$ on $\kappa(\mathrm P\Sigma_3)$ (whose elements are described as Lie polynomials
in the free Lie algebra $\mathfrak g$ on $a$, $b$ and $c$) as follows. 

Recall that $F$ denotes the subgroup of inner automorphisms of $F_3$. It is clearly stable under the action of $\Sym_3$ on $\mathrm P\Sigma_3$. As a consequence, there is an induced action on $P\Sigma_3/F$, given by $\sigma \cdot \pi(g) = \pi(\sigma \cdot g)$, where $\pi : \mathrm P\Sigma_3 \twoheadrightarrow P\Sigma_3/F$ is the canonical projection. Recall that $P\Sigma_3/F$ is free on the classes $\alpha = \pi(K_{21}) = \pi(K_{31})^{-1}$, $\beta = \pi(K_{12}) = \pi(K_{32})^{-1}$ and $\gamma = \pi(K_{13}) = \pi(K_{23})^{-1}$. Using that $\sigma \cdot K_{ij} = K_{\sigma(i)\sigma(j)}$, we can write down explicitly the action of $\sigma$ on the generators. Namely, the transposition $(12)$ acts by the automorphism exchanging $\alpha$ and $\beta$ and sending $\gamma$ to $\gamma^{-1}$, whereas $(23)$ acts by sending $\alpha$ to $\alpha^{-1}$ and exchanging $\beta$ and $\gamma$. 

By functoriality of $\Lie_*(-)$, the Lie algebra $\Lie_*(\mathrm P\Sigma_3/F)$ inherits an action of $\Sym_3$ by automorphism with respect to which the canonical projection $\Lie_*(\pi)$ is equivariant. Explicitly,  $\Lie_*(P\Sigma_3/F)$ is free on $\overline\alpha$, $\overline\beta$ and $\overline\gamma$, $(12)$ acts by exchanging $\overline\alpha$ and $\overline\beta$ and sending $\overline\gamma$ to $- \overline\gamma$, whereas $(23)$ acts by sending $\overline\alpha$ to $-\overline\alpha$ and exchanging $\overline\beta$ and $\overline\gamma$. Moreover, recall that $\Lie_*(\pi)$ induces an isomorphism (sending $a$, $b$ and $c$ to $\overline\alpha$, $\overline\beta$ and $\overline\gamma$ respectively) when restricted to~$\mathfrak g$. In particular, its restriction to $\kappa(\mathrm P\Sigma_3)$ is injective and $\Sym_3$-equivariant. This allows us to compute the action of elements of $\kappa(\mathrm P\Sigma_3)$, expressed in terms of $a$, $b$ and $c$, by identifying $a$ (resp.~$b$, $c$) with $\overline\alpha$ (resp.~$\overline\beta$, $\overline\gamma$). For instance, one can easily check that $\tau \cdot \varpi = -\varpi$ for $\tau = (12)$ and $\tau = (23)$, so $\kappa_6(\mathrm P\Sigma_3) = \Z\varpi$ is the sign representation of $\Sym_3$.

\section{Final remarks}

Explicit computer calculations (of $\tau|_{\mathfrak g}$) allow us to find explicit generators of $\kappa_k(\mathrm{P}\Sigma_3) \otimes \Q$ for $k \leq 9$. The corresponding dimensions are as follows (where we also give the dimension of the ambient space $\mathbb L_k[3] \otimes \Q$):
\[
\begin{array}{c|ccccccccc}
k & 1 & 2 & 3 & 4 & 5 & 6 & 7 & 8 & 9\\
\hline
\dim_{\Q}\mathbb L_k[3] & 3 & 3 & 8 & 18 & 48 & 116 & 312 & 810 & 2184\\
\dim_{\Q}(\kappa_k(\mathrm{P}\Sigma_3)) & 0 & 0 & 0 & 0 & 0 & 1 & 6 & 24 & 92
\end{array}
\]

Since $\mathbb L_k[3]$ is torsion-free, $\dim_{\Q}(\kappa_k(\mathrm{P}\Sigma_3)) = 0$ is enough to conclude that $\kappa_k(\mathrm{P}\Sigma_3) = 0$, so this holds for $k \leq 5$. Then Proposition~\ref{eq_JohnsonKernel} ensures that for any $k \leq 6$
\[ \Gamma_k(\mathrm{P}\Sigma_3) = \mathcal{A}_k(F_3)\cap \mathrm{P}\Sigma_3.\]

As a consequence, $\mathcal{A}_7(F_3)\cap \mathrm{P}\Sigma_3 \subseteq \mathcal{A}_6(F_3)\cap \mathrm{P}\Sigma_3 = \Gamma_6(\mathrm{P}\Sigma_3)$, so the quotient $(\mathcal{A}_7(F_3)\cap \mathrm{P}\Sigma_3)/\Gamma_6(\mathrm{P}\Sigma_3)$ is a subgroup of $\Lie_6(\mathrm{P}\Sigma_3)$, hence it is in fact torsion-free abelian.

\medskip

Our computations give explicit generators of the $\kappa_k(\mathrm{P}\Sigma_3)$ for $k = 6,7,8,9$, and they allow us to describe these  as $\Sym_3$-representations. Precisely, let us write $\widetilde H$ for $\Lie_*(P\Sigma_3/F)$, seen as a $\Sym_3$-representation (isomorphic to $H \otimes \mathrm{sgn}$).  We find that the map $\widetilde H \otimes \kappa_k(\mathrm{P}\Sigma_3) \rightarrow \kappa_{k+1}(\mathrm{P}\Sigma_3)$ induced by the Lie bracket is injective for $k \leq 8$, but not surjective if $k = 5, 6, 7, 8$. This means that we find new generators for the ideal $\kappa_k(\mathrm{P}\Sigma_3)$ in degrees $k = 6,7,8,9$. We do not know if this ideal (which coincides with the ideal of relations for the Lie subalgebra of $\Der(\mathbb L[3])$ generated by $d_{12}$, $d_{21}$ and $d_{13}$) is finitely generated. We sum up the results of these calculations in the following table (where a character $\chi$ is given as $(\chi(id), \chi((12)), \chi((123)))$) :
\[
\begin{array}{c|c|c|c|c}
k & 6 & 7 & 8 & 9\\
\hline
\text{character of }\kappa_k(\mathrm{P}\Sigma_3) & (1, -1, 1) & (6,0,0)  & (24, -2,0) & (92, 0, 2)
\end{array}
\]

\subsection*{Acknowledgements}
The second author is supported by JSPS KAKENHI Grant Number 18K03204.
The third author is supported by JSPS KAKENHI Grant Number 22K03299.

\end{document}